\documentclass{amsart}
\usepackage{latexsym, bbm, enumerate, amssymb, amsmath}
\usepackage[dvips]{graphicx}
\usepackage{setspace}
\usepackage{paralist}
\usepackage{ifpdf}

\usepackage[show]{ed}

\usepackage{tikz}

\theoremstyle{plain}
\newtheorem{theorem}{Theorem}

\newtheorem{lemma}[theorem]{Lemma}
\newtheorem{proposition}[theorem]{Proposition}

\theoremstyle{definition}

\allowdisplaybreaks[1]

\newcommand{\F}{\mathcal{F}}

\newcommand{\bxi}{{\boldsymbol{\xi}}}
\newcommand{\bM}{{\boldsymbol{M}}}
\newcommand{\bEta}{{\boldsymbol{\eta}}}

\let\oldmarginpar\marginpar
\renewcommand\marginpar[1]{\-\oldmarginpar[\raggedleft\footnotesize #1]%
{\raggedright\footnotesize #1}}

\begin{document}

\title[Martingales and Permutations]%
{A Systematic Martingale Construction with Applications to Permutation Inequalities}

\author[Pozdnyakov and Steele]%
{Vladimir Pozdnyakov and J. Michael Steele}

\thanks{Vladimir Pozdnyakov: Department of Statistics,
University of Connecticut, 215 Glenbrook Road, U-4120, Storrs, CT 06269-4120, Email: \texttt{vladimir.pozdnyakov@uconn.edu}}

\thanks{J. M. Steele: The Wharton School, Department of Statistics, Huntsman Hall
447, University of Pennsylvania, Philadelphia, PA 19104.
Email address: \texttt{steele@wharton.upenn.edu}}

\begin{abstract}
We illustrate a process that constructs martingales from raw material that arises naturally from the theory of
sampling without replacement.
The usefulness of the new martingales is illustrated by the development of maximal
inequalities for permuted sequences of real numbers.
Some of these inequalities are new and some are variations of classical inequalities like those introduced by A. Garsia in the study of
rearrangement of orthogonal series.


\end{abstract}
\date{\today}


\maketitle

\section{Introduction: A Motivating Question}

Let ${\mathcal X}=\{x_1,x_2,...,x_n\}$ be a fixed set of real
numbers, and let  $X_1,X_2,...,X_n$ be the successive values of a sample of size
$n$ that is drawn sequentially without replacement from the set $\mathcal X$.
We are concerned here with a systematic process by which one can construct martingales with respect to the sequence of sigma-fields
$\sigma (X_1, X_2, ...,X_k)$, $1 \leq k \leq n$.
For example, we consider the partial sums
$$
S_k=X_1+X_2+ \cdots+X_k, \quad 1 \leq k \leq n,
$$
and ask: Is there an $\{\F_k: 1 \leq k \leq n\}$ martingale where the values $\{S_k^2: 1 \leq k \leq n\}$ appear
in an simple and explicit way? What about the values $\{S_k^3: 1 \leq k \leq n\}$, or the partial sums of $X_i^2$, etc.
We show that there is a practical, unified approach to these problems. It faces some limitations due to the burden of algebra,
but one can make considerable progress before those burdens become too cumbersome.

We first illustrate the construction with two basic examples. These lead in turn to several martingales whose usefulness we indicate by
the derivation of permutation inequalities ---
both old and new.

\section{First Example of the Construction}

To begin, we consider $T_k=X_1^2+X_2^2+ \cdots+X_k^2$ for $1\leq k\leq n$, and we ask for a martingale where $T_k$ (or a deterministic
multiple of $T_k$) appears as a summand. The set $\mathcal X$ is known before sampling beings, and
the only source of randomness is the sampling process itself. The totals are known, deterministic, values which we denote by
\begin{equation}
M=S_n \quad  \text{and} \quad B=T_n.
\end{equation}
We first compute the conditional expectation of $S^2_{k+1}$ given the immediate past,
\begin{align*}
    E[S_{k+1}^2|\F_k]&=E[(S_{k}+X_{k+1})^2|\F_k]\\
                     &=E[(S_{k}^2+2S_kX_{k+1}+X_{k+1}^2|\F_k]\\
                     &=S_k^2+2S_k\frac{M-S_k}{n-k}+\frac{B-T_k}{n-k}\\
                     &=\frac{n-k-2}{n-k}S_k^2+\frac{2M}{n-k}S_k-\frac{1}{n-k}T_k+\frac{B}{n-k}.
\end{align*}
To organize this information, introduce the vector-column $\bxi_k=(S_k^2,S_k,T_k,1)^\top$, and note
that we also have
\begin{align*}
    E[S_{k+1}|\F_k]&=S_k+\frac{M-S_k}{n-k}=\frac{n-k-1}{n-k}S_k+\frac{M}{n-k}, \quad \text{and }\\
    E[T_{k+1}|\F_k]&=T_k+\frac{B-T_k}{n-k}=\frac{n-k-1}{n-k}T_k+\frac{B}{n-k}.
\end{align*}
These observations can be combined into one matrix equation for the one-step conditional expected values
$$E[\bxi_{k+1}|\F_k]=A_{k+1}\bxi_k,$$
where here the deterministic $4 \times 4$ matrix $A_{k+1}$ is given explicitly by
{\small
$$A_{k+1}=\left(
      \begin{array}{cccc}
      \displaystyle  \frac{n-k-2}{n-k} & \displaystyle\frac{2M}{n-k} & \displaystyle-\frac{1}{n-k} & \displaystyle\frac{B}{n-k} \\ &&&\\
      \displaystyle 0 & \displaystyle\frac{n-k-1}{n-k} & 0 & \displaystyle\frac{M}{n-k} \\ &&&\\
      \displaystyle 0 & 0 & \displaystyle\frac{n-k-1}{n-k} & \displaystyle\frac{B}{n-k} \\ &&&\\
      \displaystyle 0 & 0 & 0 & 1 \\
      \end{array}
    \right).
$$}
The matrices $\{A_k: 1 \leq k \leq n-2\}$ are invertible and deterministic, so the vector process
\begin{equation}\label{eq:MartingaleRepr}
\bM_k=A_1^{-1}A_2^{-1}\cdots A_k^{-1}\bxi_k
\end{equation}
is well-defined and adapted to $\{\F_k: 1 \leq k \leq n-2 \}$. To check that
it is a vector martingale we only need to note that
\begin{align}
E(\bM_{k+1}|\F_k)=&E(A_1^{-1}A_2^{-1}\cdots A_{k+1}^{-1}\bxi_{k+1}|\F_k) \label{ew:GeneralProcedure} \\
                             =&A_1^{-1}A_2^{-1}\cdots A_{k+1}^{-1}E(\bxi_{k+1}|\F_k) \notag \\
                             =&A_1^{-1}A_2^{-1}\cdots A_{k+1}^{-1}A_{k+1}\bxi_k
                             =A_1^{-1}A_2^{-1}\cdots A_k^{-1}\bxi_k
                             =\bM_k. \notag
\end{align}
Now, to extract the benefit from the martingale $\{\bM_k\}$, one just needs to make it more explicit.

Here one is fortunate; an easy induction confirms that $A_1^{-1}A_2^{-1}\cdots A_k^{-1}$ is given by
the  upper-triangular matrix:
{\small
$$\frac{1}{n-k}\left(
      \begin{array}{cccc}
      \displaystyle  \frac{n(n-1)}{n-k-1} & \displaystyle-\frac{2knM}{n-k-1} & \displaystyle\frac{kn}{n-k-1}& \displaystyle\frac{k(k+1)M^2-knB}{n-k-1 }\\ &&&\\
      \displaystyle 0 & \displaystyle n & 0 & \displaystyle-kM \\ &&&\\
      \displaystyle 0 & 0 & \displaystyle n & \displaystyle-kB \\ &&&\\
      \displaystyle 0 & 0 & 0 & n-k \\
      \end{array}
    \right).
$$}
Now, the coordinates of the vector martingale $\bM_k= (M_{1,k},M_{2,k},M_{3,k},M_{4,k})^\top$ are martingales in their own
right, and it is worthwhile to examine them individually.

The fourth coordinate just gives the trivial martingale  $M_{4,k}\equiv 1$, but the other coordinates are much more interesting.
The second and third coordinates give us two useful --- but known --- martingales,
\begin{equation}\label{eq:MartingalsM2M3}
M_{2,k}=(nS_k-kM)/(n-k) \quad \text{and } M_{3,k}=(nT_k-kB)/(n-k).
\end{equation}
Actually, we have only one martingale here; one gets the martingale $\{M_{3,k}\}$ from
the martingale $\{M_{2,k}\}$ if one replaces ${\mathcal X}$  with the set of squares ${\mathcal X'}=\{x_1^2,x_2^2,...,x_n^2\}$.

The martingale $\{M_{2,k}\}$ is given in Serfling (1974) and Stout (1974, p.~147).
The earliest source we could identify for the martingale
$\{M_{2,k}\}$ is Garsia (1968,~p.~82), but it is hard to say if this was its first appearance.
The martingale $\{M_{2,k}\}$ is a natural one with few impediments to its discovery.

To get a martingale that is more novel (and less transparent), we only need to consider the
first coordinate of the vector martingale $\bM_k$. One can write this coordinate explicitly as
\begin{align}
   M_{1,k}=&\frac{n(n-1)}{(n-k)(n-k-1)}S_k^2-\frac{2knM}{(n-k)(n-k-1)}S_k \notag \\
          &+\frac{kn}{(n-k)(n-k-1)}T_k+\frac{k(k+1)M^2-knB}{(n-k)(n-k-1)}. \notag
\end{align}
This may seem complicated at first sight, but there is room for simplification.
In many situations it is natural to  assume that $M=S_n=0$,
and, in that case, $M_{1,k}$ reduces to the more manageable martingale sequence that we denote by
\begin{equation}\label{eq:FirstNewMartingale}
\widetilde{M}_k=\frac{1}{(n-k)(n-k-1)}\left[(n-1)S_k^2-(B-T_k)k\right], \quad \quad 1\leq k \leq n-2.
\end{equation}

We will give several applications of this martingale to permutation inequalities, and, in particular, we
use it in Section \ref{sec:QuadraticRearrangement} to get a new bi-quadratic maximal inequality for permuted arrangements.
We should also note that this martingale also has a potentially useful monotonicity property. Specifically,
$T_k$ is monotone increasing, so by a little surgery on $T_k$ (say by replacing $T_k$ by $k^\alpha T_k$)
will yield one a rich family of submartingales or supermartingales.

The device used here to get the 4-vector martingale $\{\bM_n: 1 \leq k \leq n-2 \}$ can be extended in several ways.
The most direct approach begins with $S_k^3$ in addition to $S_k^2$. In that case, linearization of the recursions
requires one to introduce the terms
$S_kT_k$ and $U_k=X_1^3+X_2^3+ \cdots+X_k^3$. The vectors $\{\bxi_k\}$ are then 7-dimensional, and,
the matrix algebra is still tractable through symbolic computation, but it is unpleasant to display.
One can follow the construction
and obtain seven martingales. Some of these are known, but many as four may be new.
The ones with $S_k^3$ and $S_kT_k$ are (almost) guaranteed to be new.

Nevertheless, we do not pursue the seven-dimensional example here. Instead, for our second example, we apply
the general construction to a simpler two-dimensional problem.
In this example, the algebra is much more attractive, and the martingale that one finds has
considerable flexibility.

\section{Second Example of the Construction}\label{sec:SecondConstruction}

As before, we assume that $X_1, X_2, ...,X_n$ is a sequential sample  taken
without replacement from $\mathcal X$, and we impose the
centering condition $S_n=M=0$. Further, we consider a system of
``multipliers" $a_k$, $1\leq k\leq n$ where each $a_k$ is assumed to be an $\mathcal{F}_{k-1}$ measurable
random variable.
Since our measure space is finite, these non-anticipating random variables are automatically bounded.

The basic building block for our next collection of martingales
is the sequence of random variables defined by setting
$$
W_k=a_1X_1+a_2X_2+\cdots+a_kX_k \quad \text{for }1\leq k\leq n.
$$
The immediate task is to find a martingale that has $W_k$ (or a deterministic multiple of $W_k$) as a summand.

As before, we begin by calculating the one-step conditional expectations:
$$
E[W_{k+1}|\mathcal{F}_k]=W_k-\frac{a_{k+1}}{n-k}S_k \quad \text{and } \quad E[S_{k+1}|\F_k]=\frac{n-k-1}{n-k}S_k.
$$
If we introduce the vector $\bEta_k=(W_k,S_k)^{\top}$ we have
$$
E[\bEta_{k+1}|\mathcal{F}_k]=A_{k+1}\bEta_k
\quad \text{where }
{
A_{k+1}=\left(
      \begin{array}{cc}
      \displaystyle  1 & \displaystyle -\frac{a_{k+1}}{n-k} \\ &\\
      \displaystyle 0 &\displaystyle \frac{n-k-1}{n-k}
  \end{array}
    \right)}.
$$
Inversion is now especially easy, and we note that for
$$A_{k}=\left(
      \begin{array}{cc}
      \displaystyle  1 & \displaystyle -\frac{a_{k}}{n-k+1} \\ &\\
      \displaystyle 0 & \displaystyle \frac{n-k}{n-k+1}
  \end{array}
    \right)
\quad \text{we have } \quad
A_{k}^{-1}=\left(
      \begin{array}{cc}
      \displaystyle  1 & \displaystyle \frac{a_{k}}{n-k} \\ &\\
      \displaystyle 0 & \displaystyle \frac{n-k+1}{n-k}
  \end{array}
    \right),
$$
so induction again confirms the critical inverse:
$$A_1^{-1}A_2^{-1}\cdots A_{k}^{-1}=\left(
      \begin{array}{cc}
      \displaystyle  1 & \displaystyle \frac{a_1+a_2+\cdots+a_{k}}{n-k} \\ &\\
      \displaystyle 0 & \displaystyle \frac{n}{n-k}
  \end{array}
    \right).$$
The general recipe \eqref{ew:GeneralProcedure} then gives us a new martingale:
\begin{equation}\label{eq:WeightedMartingale}
M_k= W_k+\frac{a_1+a_2+ \cdots +a_k}{n-k}S_k \quad \quad 1\leq k < n
\end{equation}
Once this is martingale is written down, one could also verify the martingale property by a direct calculation.
In this instance, linear algebra has served us mainly as a tool for discovery.

Some of the benefits of the martingale \eqref{eq:WeightedMartingale} are brought to life through interesting choices for the
non-anticipating factors $\{a_k: 1\leq k \leq n\}$.
For, example, if we take  $a_1=0$ and set $a_k=X_{k-1}$ for $k\geq 2$, we find a curious quadratic martingale
$$
M_k= X_1X_2+X_2X_3+\cdots+X_{k-1}X_k+\frac{(X_1+\cdots+X_{k-1})(X_1+\cdots+X_k)}{n-k}.
$$
One is unlikely to have hit upon this martingale without the benefit of a systematic construction.
For the moment, no application of this martingale comes to mind, but it does seem useful to know that
there is such a simple quadratic martingale. Perhaps a nice application is not too far away.

Out main purpose here is to expose the general construction \eqref{ew:GeneralProcedure}, but, through the illustrations
just given, we now have several martingales that speak directly to permutation inequalities --- an extensive subject
where martingales have traditionally been part of the toolkit. As we explore
what can be done with the new martingales \eqref{eq:FirstNewMartingale} and \eqref{eq:WeightedMartingale},
we will also address some of the classic results on permutation inequalities.

\section{A Permutation Inequality}\label{sec:firstRAI}

In the first application, we quickly check
what can be done with the martingale $M_{2,k}=(nS_k-kM)/(n-k)$ given by \eqref{eq:MartingalsM2M3} in our first construction.
To keep the formulas simple, we impose a standing assumption,
\begin{equation}\label{M=0}S_n=M=0,\end{equation}
so,
$
{S_k}/({n-k})
$
is an $\{\F_k \}$ martingale with expectation zero.
The Doob-Kolmogorov $L^2$ maximal inequality (see e.g.~Shiryaev (1995, p.~493)) then gives us
$$E\max_{1\leq k\leq n-1}\left|({n-k})^{-1}{S_k}\right|^2\leq 4 ES_{n-1}^2.$$
This can be simplified by noting that
$$
ES_{n-1}^2=E(M-X_n)^2=EX_n^2=EX_1^2={B}/{n},
$$
so, in the end, we have
\begin{equation}\label{easy*Garsia}
E\max_{1\leq k\leq n-1}\left|({n-k})^{-1}{S_k}\right|^2\leq 4B/{n}.
\end{equation}
One could immediately transcribe this as a permutation inequality, but first we put it in a form that seems more natural
for applications.

For any fixed permutation $\sigma$, the distribution of the vector $(X_1, X_2,\dots, X_n)$ is
the same as distribution the vector $(X_{\sigma(1)}, X_{\sigma(2)}, \ldots ,X_{\sigma(n)})$ (cf.~Feller (1971, p.~228)),
so, in particular the distribution of $(X_1, X_2,\dots, X_n)$ is the same as the distribution of $(X_n, X_{n-1},\dots, X_1)$.
By the centering assumption \eqref{M=0}, we know
$$
(X_1+X_2+\cdots+X_k)^2=(X_{k+1}+X_{k+2}+\cdots +X_n)^2,
$$
so applying both observations gives us the identity
$$E\max_{1\leq k\leq n-1}\left|({n-k})^{-1}{S_k}\right|^2=E\max_{1\leq k\leq n-1}\left|{S_k}/{k}\right|^2.$$
Using this bound in \eqref{easy*Garsia} then gives us
\begin{equation}\label{eq:M2Identity}
E\max_{1\leq k\leq n-1}\left|{S_k}/{k}\right|^2\leq {4B}/{n},
\end{equation}
so by the sampling model and the definition of $B$, we come to an attractive inequality
for the maximum of averages drawn sequentially from a randomly permutated sample.

\begin{proposition}[Max-Averages Inequality]\label{Serfling*Martingale}
For real numbers $\{x_1,x_2,...,x_n\}$ with $x_1+x_2+\cdots+x_n=0$, we have
\begin{equation}\label{eq:SerflingInequality}
\frac{1}{n!} \sum_{\sigma}\max_{1\leq k\leq n}\left\{ \frac{1}{k} {\sum_{i=1}^kx_{\sigma(i)}}\right\}^2
\leq \frac{4}{n}\sum_{i=1}^nx_i^2.
\end{equation}
\end{proposition}

This bound is so natural, it is likely to be part of the folklore of permutation inequalities, but we have been unable to locate it in
earlier work. Still, even if the inequality is known, it seems probable that it has been under appreciated. In two examples in
Section \ref{sec:AlternaitingSums} we find that it provides an efficient alternative to other, more complicated tools.

This inequality also bears a family relationship to an important inequality that originated with Hardy (1920).
Hardy's inequality went through some evolution before it reached its modern form (see Steele (2004), pp.~169 and 290 for historical
comments), but now one may write Hardy's inequality in a definitive way that underscores the analogy with \eqref{eq:SerflingInequality}:
\begin{equation}\label{eq:Hardy}
\max_{\sigma} \sum_{k=1}^n \left\{ \frac{1}{k} {\sum_{i=1}^kx_{\sigma(i)}}\right\}^2
\leq 4 \sum_{i=1}^n x_i^2.
\end{equation}
It is well known (and easy to prove) that the constant in Hardy's inequality is best possible, and it is feasible  that
the constant in \eqref{eq:SerflingInequality} is also best possible. We have not been able to
resolve this question.

\section{Exchangeability and a Folding Device}\label{sec:foldingdevice}

One also has the possibility of using exchangeability more forcefully;
in particular, one can exploit exchangeability around the center of the sample.
For a preliminary illustration of this possibility, we fix $1\leq m < n$ and note that
the martingale property \eqref{eq:FirstNewMartingale} of $\widetilde{M}_k$ gives us
\begin{align*}
ES_m^2=&E\frac{m}{n-1}[B-T_m] =\frac{m}{n-1}[B-ET_m].
\end{align*}
The martingale property \eqref{eq:MartingalsM2M3} of $M_{3,k}$ also gives us
$ET_m=({m}/{n})B$, so we have
\begin{equation}\label{S_2*Expectation}
ES_m^2=\frac{m(n-m)}{n(n-1)}B,
\end{equation}
a fact that one can also get by bare hands, though perhaps not so transparently.

We can also use the martingale
$M_{2,k}=(nS_k-kM)/(n-k)= nS_k/(n-k)$ here. Since we have $M=0$, the $L^2$ maximal inequality for this martingale
and the identity \eqref{S_2*Expectation} give us
\begin{equation}\label{eq:preFold}
E\max_{1\leq k\leq m}\left|\frac{S_k}{n-k}\right|^2\leq \frac{4 ES_{m}^2}{(n-m)^2} =  \frac{4 m B}{n(n-1)(n-m)}.
\end{equation}
Since $(n-k)^{2} \leq n^{-1}(n-1)$ for $1\leq k\leq m$, the bound \eqref{eq:preFold} implies the weaker, but simpler, bound
\begin{equation}\label{eq:FirstHalfFold}
E\max_{1\leq k\leq m}S_k^2\leq \frac{4mB}{n-m} \quad \quad \text{for } 1 \leq m < n.
\end{equation}
The idea now is to use symmetry to exploit the fact that \eqref{eq:FirstHalfFold} holds for many choices of $m$.
To begin, we note that we always have the crude bound
\begin{equation}\label{eq:crude}
E\max_{1\leq k\leq n}S_k^2\leq E\max_{1\leq k\leq m}S_k^2+E\max_{m< k\leq n}S_k^2.
\end{equation}
Typically this is useless, but here it points to a useful observation; we can use the same bounds on the second terms that we used on the first.
This is the ``folding device" of the section heading.

Specifically, since $M=0$ we have
$$
|S_k|=|S_n-S_k|=|X_n+X_{n-1}+\cdots+X_{k+1}|,
$$
so, by exchangeability, we see that the second summand of \eqref{eq:crude} is bounded by $4(n-m)B/m$. Thus, we have for
all $1\leq m < n$ that
\begin{equation}
E\max_{1\leq k\leq n}S_k^2\leq E\max_{1\leq k\leq m}S_k^2+E\max_{m< k\leq n}S_k^2\leq 4B \left(\frac{m}{n-m}+\frac{n-m}{m}\right).
\end{equation}
Here,  if we just consider  $n\geq 8$ and choose $m$ as close as possible to $n/2$, then
we see that the worst case occurs when $n=9$ and $m=4$; so we have the bound
\begin{equation}\label{backward*Serfling}
E\max_{1\leq k\leq n}S_k^2\leq E\max_{1\leq k\leq m}S_k^2+E\max_{m< k\leq n}S_k^2\leq (41/5)B \quad \text{for } n \geq 8.
\end{equation}

On the other hand, Cauchy's inequality gives us $S_k^2\leq kB \leq nB$, so the bound
\eqref{backward*Serfling} is trivially true for $n\leq 7$.
Combining these ranges gives us a centralized version of an inequality originating with A. Garsia (cf. Stout (1974, pp. 145--148)).

\begin{proposition}[A. Garsia] \label{Garsia*Inequality}
For any set of real numbers $\{x_1,x_2,...,x_n\}$ with sum $x_1+x_2+\cdots+x_n=0$ we have
\begin{equation}\label{eq:GarsiaProp}
\frac{1}{n!}\sum_{\sigma}\max_{1\leq k\leq n}|\sum_{i=1}^kx_{\sigma(i)}|^2\leq (8+\frac{1}{5})\sum_{i=1}^nx_i^2,
\end{equation}
where the sum is taken over all possible permutations of $\{x_1,x_2,...,x_n\}$.
\end{proposition}

Here we have paid some attention to the constant in this inequality, but already from the crude bound
\eqref{eq:crude} one knows that the present approach is not suited to the derivation of a best possible result.
Our intention has
been simply to illustrate what one can do with reasonable care and robust approach that uses the tools provided
by our general martingale construction.

Nevertheless, the constant in this inequality has an interesting history. The bound
seems first to have appeared in  Garsia (1968, Theorem 3.2) with
the constant $9$. Curiously, the inequality appears later in Garsia (1970, eqn. 3.7.15, p.~91) where the constant is given as $16$,
and for the proof of the inequality one is advised to ``following the same steps" of Garsia (1964). In each instance, the intended
applications did not call for sharp constants, so these variations are scientifically inconsequential. Still, they do make one curious.

Currently the best value for the constant in the Garsia inequality \eqref{eq:GarsiaProp} is due to Chobanyan (1994, Corollary 3.3)
where the stunning value of $2$ is obtained. Moreover, Chobanyan and Salehi (2001, Corollary 2.8)) have a much
more general inequality which also gives a constant of $2$ when specialized to our situation.

Here we should also note that in all
inequalities of this type have both a centralized version where $M=0$ and
non-centralized version where $M$ is unconstrained. One can pass easily between the versions (see e.g.
Stout (1974, p.~147) or Garsia (1970, p.~93)). The centralized versions are inevitably simpler to state, so we have
omitted discussions of the non-centralized versions.

\section{Quadratic Permutation Inequality}\label{sec:QuadraticRearrangement}

To apply the $L^2$ maximal inequality to the martingale $\{ \widetilde{M}_k : 1 \leq k \leq n-2 \}$
that was discovered by  our first construction \eqref{eq:FirstNewMartingale}, one needs a comfortable estimate of the
bounding term $4 E[\widetilde{M}_{n-2}^2]$ in the Doob-Kolmogorov inequality. There are
classical formulas for the moments for sampling without replacement that simplify this task. These
formulas were known to Isserlis (1931), if not before, but they are perhaps easiest to derive on one's own.

\begin{lemma}\label{Moments}
If $X_1,X_2,X_3,X_4$ are draw without replacement from ${\mathcal X}=\{x_1,x_2,...,x_n\}$  where
$x_1+x_2+ \cdots+x_n=0$, then we have the moments
\begin{align*}
&E(X_1X_2X_3X_4)=\frac{3B^2-6Q}{n(n-1)(n-2)(n-3)},\quad E(X_1^2X_2X_3)=\frac{2Q-B^2}{n(n-1)(n-2)},\\
&E(X_1^2X_2^2)=\frac{B^2-Q}{n(n-1)},\quad
E(X_1^3X_2)=-\frac{Q}{n(n-1)}, \quad \text{and} \quad E(X_1^4)={Q}/{n},
\end{align*}
where, as usual, we set $B=x_1^2+x_2^2+\cdots+x_n^2$ and $Q=x_1^4+x_2^4+ \cdots+x_n^4$.
\end{lemma}

Now, to calculate $4 E[\widetilde{M}_{n-2}^2]$, we first note that
$$
S_{n-1}=-X_n \quad \text{and} \quad S_{n-2}=-X_{n-1}-X_n,
$$
 so just expanding the
definition of $\widetilde{M}_{n-2}$ gives us
\begin{align*}
4 E[\widetilde{M}_{n-2}^2] = & E\left[\left\{(n-1)S_{n-2}^2-(n-2)(B-T_{n-2})\right\}^2\right]\\
 = & E\left[\left\{(n-1)(X_{n-1}+X_n)^2-(n-2)(X_{n-1}^2+X_n^2)\right\}^2\right]\\
 = & E\big[X_{n-1}^4+X_n^4+(4n^2-8n+6)X_{n-1}^2X_n^2\\
 &\quad +4(n-1)X_{n-1}^3X_n+4(n-1)X_{n-1}X_n^3\big].
 \end{align*}
By Lemma \ref{Moments} and exchangeability, one then finds after some algebra that
the $L^2$ maximal inequality takes the form
\begin{equation}\label{eq:MaxEq4us}
E\max_{1\leq k\leq n-2}\widetilde{M}_k^2  \leq 4 E[\widetilde{M}_{n-2}^2]=\frac{4n^2-8n+6}{n(n-1)}B^2-\frac{4n^2-2n}{n(n-1)}Q.
\end{equation}
The only task left is to reframe this inequality so that it easy to apply as a permutation inequality.
Here it is useful to observe that
\begin{equation}\label{eq:nbounds}
\frac{4n^2-8n+6}{n(n-1)}<4 \quad \text{and} \quad \frac{4n^2-2n}{n(n-1)}>4 \quad \quad \text{for } n\geq 2,
\end{equation}
moreover, these are not wasteful bounds; they are essentially sharp for large $n$. When we apply these bounds
in \eqref{eq:MaxEq4us} we get the nicer bound,

\begin{equation}\label{our*forward}
E\max_{1\leq k\leq n-2}\left|\frac{(n-1)S_k^2-k(B-T_k)}{(n-k)(n-k-1)}\right|^2 \leq 4[B^2-Q].
\end{equation}

This brings us closer to our goal, but further simplification is possible if we note that by reverse sampling the left side of
\eqref{our*forward} can also be written as
\begin{align*}E\max_{1\leq k\leq n-2}\left|\frac{(n-1)S_k^2-k(B-T_k)}{(n-k)(n-k-1)}\right|^2
=&E\max_{2\leq k\leq n-1}\left|\frac{(n-1)S_k^2-(n-k)T_k}{k(k-1)}\right|^2\\
=&E\max_{2\leq k\leq n}\left|\frac{(n-1)S_k^2-(n-k)T_k}{k(k-1)}\right|^2.
\end{align*}
In permutation terms, the bound \eqref{our*forward} establishes the following proposition.

\begin{proposition}[Quadratic Permutation Inequality]\label{our*martingale}
For any set of real numbers $\{x_1,x_2,...,x_n\}$ with $x_1+x_2+\cdots+x_n=0$ we have
\begin{equation*}
\frac{1}{n!}\sum_{\sigma}\! \max_{2\leq k\leq n}\!\left|\frac{\left(\sum_{i=1}^kx_{\sigma(i)}\right)^2
\!\!-\frac{n-k}{n-1}\sum_{i=1}^kx_{\sigma(i)}^2}{k(k-1)}\right|^2
\!\leq \frac{4}{(n-1)^2} \! \left\{\left(\sum_{i=1}^nx_i^2\right)^2-\sum_{i=1}^n x_i^4 \right\}\!,
\end{equation*}
where the first sum is taken over all possible permutations of $\{x_1,x_2,...,x_n\}$.
\end{proposition}

At first glance, this may seem complicated, but the components are all readily interpretable and the inequality
is no more complicated than it has to be. The main observation is that our  general construction \eqref{eq:MartingaleRepr}
brought us here
in a completely straightforward way. Without such an on-ramp, one is unlikely to imagine
any bound of this kind --- and \emph{relative} simplicity. Moreover, the inequality does have
intuitive content, and this content is made more explicit in the next section.

\section{The Discrete Bridge and Further Folding}\label{sec:DiscreteBridge}

Here we set $x_i=1$ for $1 \leq i \leq m$ and take $x_i=-1$ for $m+1 \leq i \leq 2m$. We then
let $X_i$, $1\leq i \leq  2m$, denote samples that are
drawn without replacement from the set ${\mathcal X}=\{x_1,x_2, \ldots, x_{2m}\}$.
If we put $S_0=0$ and denote the usual partial sums by $S_k$, $1\leq k \leq 2m$, then $S_{2m}=0$ and
the process $\{S_k: 0 \leq k \leq 2m\}$ is a discrete analog of the Brownian bridge.
Alternatively, one can view this process as simple random walk that is conditioned to return to $0$ at time $2m$.

For $\{\widetilde{M}_k\}$, the martingale from the first construction \eqref{eq:FirstNewMartingale}, we now have
$T_k\equiv k$, so we have the simple representation
\begin{equation}\label{eq:DBMart}
\widetilde{M}_k=\frac{(2m-1)S_k^2-(2m-k)k}{(2m-k)(2m-k-1)}, \quad \quad 1\leq k\leq 2m-2.
\end{equation}
By Lemma~\ref{Moments} with $B=Q=2m$ and by \eqref{S_2*Expectation},
we also find that at the (left) mid-point $m$ of our process we have the nice relations
\begin{equation}\label{eq:ExpS4}
E[S^2_m]=m^2/(2m-1)\quad \text{and } E[S_m^4]=\frac{3m^4-4m^3}{4m^2-8m+3}.
\end{equation}
These give us a rational formula for $E[\widetilde{M}_m^2]$, but for the moment, we just use
the partial simplification
\begin{align*}
E[\widetilde{M}_m^2]&=E\left|\frac{(2m-1)S_{m}^2-(2m-m)m}{(2m-m)(2m-m-1)}\right|^2\\
&= \frac{1}{m^2(m-1)^2}\left[(2m-1)^2ES_{m}^4-m^4\right].
\end{align*}
For $1\leq k \leq m$ we have the trivial bound
$$
1\leq \max_{1\leq k\leq m}\left|\frac{(2m)(2(m-1))}{(2m-k)(2m-k-1)}\right|^2,
$$
so the $L^2$ maximal inequality
applied to the martingale \eqref{eq:DBMart} gives us
\begin{equation}\label{eq:ForBoth}
E\max_{1\leq k\leq m}\left|(2m-1)S_{k}^2-(2m-k)k\right|^2 \leq 64\left[(2m-1)^2ES_{m}^4-m^4\right].
\end{equation}

We can now take advantage of a symmetry that is special to the discrete bridge. Since  $S_{2m}=0$ we have
$(X_1+\cdots+X_k)^2=(X_{2m}+\cdots+X_{k+1})^2$, and by exchangeability the vectors
$(X_1,\dots,X_m)$ and $(X_{2m},\dots,X_{m+1})$ have the same distribution.  Also, the
value $(2m-k)k$ is ``invariant" in the following sense: if we substitute for $k$ (the number of summands in $X_1+X_2+\cdots+X_k$)
the value $2m-k$ (the number of summands in $X_{2m}+X_{2m-1}+...+X_{k+1}$), then  the symmetric quantity $(2m-k)k$ is unchanged.
As a consequence, the random variables
$$
\max_{1\leq k\leq m}\left|(2m-1)S_{k}^2-(2m-k)k\right|^2 \quad \text{and} \quad
\max_{m\leq k\leq 2m-1}\left|(2m-1)S_{k}^2-(2m-k)k\right|^2
$$
are equal in distribution.

We can then apply \eqref{eq:ForBoth} twice to obtain
\begin{align*}
E\max_{1\leq k\leq 2m-1}\left|(2m-1)S_{k}^2-(2m-k)k\right|^2\leq 128\left[(2m-1)^2ES_{m}^4-m^4\right].
\end{align*}
From this bound and the formula \eqref{eq:ExpS4} for $ES_m^4$, we then have
\begin{align*}
E\max_{1\leq k\leq 2m-1}\left|S_{k}^2-k\frac{2m-k}{2m-1}\right|^2\leq 128\frac{4m^3-8m^2+4m}{8m^3-20m^2+14m-3}m^2,
\end{align*}
and, for any $m\geq 2$  we have $(4m^3-8m^2+4m)/(8m^3-20m^2+14m-3)<1$. In the end, we have the following proposition.

\begin{proposition}[Quadratic Permutation Inequality for Discrete Bridge]\label{Now*done*proposition}
For the set
${\mathcal X}=\{x_1,x_2, \ldots, x_{2m}\}$
with $x_i=1$ for $1 \leq i \leq m$ and $x_i=-1$ for $m+1 \leq i \leq 2m$ one has
\begin{align}\label{Now*Done}
\frac{1}{(2m)!}\sum_{\sigma}\max_{1\leq k\leq 2m-1}\left|\left(\sum_{i=1}^kx_{\sigma(i)}\right)^2
- k\frac{2m-k}{2m-1}\right|^2\leq 128 m^2,
\end{align}
where the first sum is taken over all possible permutations.
\end{proposition}

For the terms of the squared discrete bridge process $\{S_k^2: 0\leq i \leq 2m\}$ have the expectations
$E S_{k}^2=k(2m-k)/(2m-1)$, so \eqref{Now*Done} gives us a rigorous
bound on the maximum deviation of the square $S_k^2$ of a discrete bridge from its expected value $ES_k^2$ .
Here, the order $O(m^2)$ of the bound cannot be improved, but, if necessity called, one may be able to improve on
ungainly constant $128$.

\section{Permuted Sums with Fixed Weights}\label{sec:AlternaitingSums}

In Section \ref{sec:SecondConstruction}, we considered the weighted sums
$$
W_k=a_1X_1+a_2X_2+\cdots+a_kX_k \quad \text{for }1\leq k\leq n.
$$
and we found that the process
\begin{equation}\label{eq:WeightedMartingale2}
M_k= W_k+\frac{a_1+a_2+ \cdots +a_k}{n-k}S_k \quad \quad 1\leq k < n
\end{equation}
is a martingale whenever the multipliers are non-anticipating random variables  (i.e. whenever $a_k$ is $\F_{k-1}$-measurable
for each $1\leq k \leq n$). Here we will show that this martingale has informative uses even when one simply takes the
multipliers to be fixed real numbers.

First we introduce some shorthand. For $1 \leq k \leq n$ we will write
\begin{equation}\label{eq:alphaNotation}
\alpha_1(k)=a_1+a_2+\cdots+a_k \quad \text{and} \quad \alpha_2(k)=a^2_1+a^2_2+\cdots+a^k_k.
\end{equation}
Also, we will be most interested in multiplier vectors $a=(a_1, a_2, \ldots, a_n)$ for which we have
some control of the quantity
\begin{equation}\label{eq:Vcondition}
V_n(a)\stackrel{\rm def}{=}\max_{1\leq k\leq n-1}\alpha_1^2(k)\big/\alpha_2(n),
\end{equation}
which is a measure of cancelation among the multipliers. A leading example worth keeping in mind is
the sequence $a_k=(-1)^{k+1}$, $1 \leq k \leq n$, for which we have $V_n(a)=1/n$.
We will revisit the measure $V_n(a)$ after we derive a moment bound.

\begin{lemma}\label{lm:Msqrd} For the martingale $\{M_k\}$ defined by equation \eqref{eq:WeightedMartingale2}, we have
\begin{align}\label{eq:Msqrd}
E[M_k^2]= \frac{1}{n-1} \alpha_2(k) B
           +\frac{1}{(n-1)(n-k)} \alpha_1^2(k) B.
\end{align}
\end{lemma}
\begin{proof}
Simply squaring \eqref{eq:WeightedMartingale2}, we have
$$
|M_k|^2=W_k^2+\frac{\alpha_1^2(k)}{(n-k)^2}S_k^2+2\frac{\alpha_1(k)}{n-k}W_kS_k,
$$
so we just need to find $EW_k^2$, $ES_k^2$ and $EW_kS_k$. We know $ES_k^2$ from (\ref{S_2*Expectation}),
and we have  $EX_i^2=B/n$ and $EX_iX_j=-B/(n(n-1))$ so
\begin{align*}
EW_k^2=&E\Big[\sum_{i=1}^{k}a_i^2X_i^2+\sum_{1\leq{i,j}\leq k, i\neq j}a_ia_jX_iX_j\Big]\\
      =&\alpha_2(k)\frac{B}{n}+\left[\alpha_1^2(k)-\alpha_2(k)\right]\left[-\frac{B}{n(n-1)}\right]
      =\alpha_2(k)\frac{B}{n-1}-\alpha_1^2(k)\left[\frac{B}{n(n-1)}\right].
\end{align*}
Similarly, we have
\begin{align*}
EW_kS_k=&E\Big[\sum_{i=1}^{k}a_iX_i^2+\sum_{i=1}^ka_i\sum_{1\leq{j}\leq k, j\neq i}X_iX_j\Big]\\
      =&\alpha_1(k)\frac{B}{n}-(k-1)\alpha_1(k)\left[\frac{B}{n(n-1)}\right]
      =\alpha_1(k)\frac{(n-k)B}{n(n-1)},
\end{align*}
so summing up the terms completes the proof of the lemma.
\end{proof}

From this lemma, the $L^2$ maximal inequality gives us
\begin{align}\label{eq:weightedMax}
E\max_{1\leq k\leq n-1} M_{k}^2 \leq & 4E[M_{n-1}]^2 =4 [\alpha_2(n-1) +\alpha_1^2(n-1)]\frac{B}{n-1},
\end{align}
but to extract real value from we need to relate it to the weighted sum $W_k$.
Using \eqref{eq:WeightedMartingale2} we can write $W_k$ as the difference between
$M_k$ and $\alpha_1(k){S_k}/({n-k})$, so from the trivial bound $(x+y)^2\leq 2x^2+2y^2$  we have
\begin{align*}
\max_{1\leq k\leq n-1}|W_k|^2&\leq 2\max_{1\leq k\leq n-1}|M_k|^2+2 \max_{1\leq k\leq n-1}\alpha_1^2(k)\left|\frac{S_k}{n-k}\right|^2\\
                             &\leq 2\max_{1\leq k\leq n-1}|M_k|^2+2\max_{1\leq k\leq n-1}\alpha_1^2(k)
                             \cdot\max_{1\leq k\leq n-1}\left|\frac{S_k}{n-k}\right|^2.
\end{align*}
The second step may look wasteful, but now we can apply both \eqref{eq:weightedMax}
and the Max-Averages inequality~(\ref{easy*Garsia})
from Section \ref{sec:firstRAI} to obtain
\begin{align*}
E\max_{1\leq k\leq n-1}|W_k|^2\leq
8\Big[\alpha_2(n-1)+\alpha_1^2(n-1) \Big]\frac{B}{n-1}
+8\max_{1\leq k\leq n-1} \alpha_1^2(k) \frac{B}{n}.
\end{align*}
Note that we also have
\begin{align*}
\max_{1\leq k\leq n}|W_k|^2&=\max\left(\max_{1\leq k\leq n-1}\left|W_k\right|^2, (W_{n-1}+a_nX_n)^2\right)\\
                             &\leq\max\left(\max_{1\leq k\leq n-1}\left|W_k\right|^2, 2W_{n-1}^2+2a_n^2X_n^2\right)\\
                             &\leq2\max_{1\leq k\leq n-1}\left|W_k\right|^2+2a_n^2X_n^2,
\end{align*}
and the bottom line is that
\begin{equation}\label{eq:quantitytobound}
E\left\{\max_{1\leq k\leq n}|W_k|^2\right\} =\frac{1}{n!}\sum_{\sigma}\max_{1\leq k\leq n}\left|\sum_{i=1}^k a_i x_{\sigma(i)}\right|^2
\end{equation}
is bounded by the lengthy (but perfectly tractable) sum
\begin{equation}\label{weighted*Garsia}
16\Big[\alpha_2(n-1)+\alpha_1^2(n-1)\Big]\frac{B}{n-1}
+16\max_{1\leq k\leq n-1}\alpha_1^2(k)\frac{B}{n}
+2a_n^2\frac{B}{n}.
\end{equation}

To make this concrete, note that for $a_k=(-1)^{k+1}$ the bound on  \eqref{eq:quantitytobound} that we get from \eqref{weighted*Garsia}
is simply $(16n/(n-1)+18/n)B$. The ratio $16n/(n-1)+18/n$ decreases to $16$, and  for $n=18$ the upper bound is equal to $(17+{16}/{17})B$.
By Cauchy's inequality, we have $|W_k|^2\leq k B$ for all $k$, so the bound  $(17+{16}/{17})B$ also holds for all $n\leq 17$.
When we assemble the pieces, we have a permutation maximal inequality for sums with alternating signs.

\begin{proposition}\label{Garsia*alternating*sums*proposition}
For real numbers  $\{x_1,x_2,...,x_n\}$ with $x_1+x_2+\cdots+x_n=0$ we have
\begin{equation}\label{Garsia*alternating*sums}
\frac{1}{n!}\sum_{\sigma}\max_{1\leq k\leq n}|\sum_{i=1}^k(-1)^ix_{\sigma(i)}|^2
\leq \left(17+\frac{16}{17}\right)\sum_{i=1}^nx_i^2
\end{equation}
where the sum is taken over all permutations of $\{x_1,x_2,...,x_n\}$.
\end{proposition}

The argument that leads to  \eqref{Garsia*alternating*sums} is useful for more than just alternating sums;
it has bite whenever there is meaningful cancelation in $a=(a_1, a_2, \ldots, a_n)$. Specifically,
the bound \eqref{weighted*Garsia}  is always dominated by
$$16\alpha_2(n)+32 \!\!\! \max_{1\leq k\leq n-1}\alpha_1^2(k),$$
so our proof also gives us more general --- and potentially more applicable --- bound.

\begin{proposition}\label{Garsia*alternating*sums*proposition*two}
For sets of real numbers $\{a_1,a_2,...,a_n\}$ and  $\{x_1,x_2,...,x_n\}$ such that $x_1+x_2+\cdots+x_n=0$ we have
\begin{equation}\label{eq:GwithV}
\frac{1}{n!}\sum_{\sigma}\max_{1\leq k\leq n}|\sum_{i=1}^ka_ix_{\sigma(i)}|^2
\leq \frac{16}{n-1}\{1+2V_n(a)\} \sum_{i=1}^na_i^2\sum_{i=1}^nx_i^2,
\end{equation}
where the sum is taken over all permutations of $\{x_1,x_2,...,x_n\}$ and  where $V_n(a)$ defined in \eqref{eq:Vcondition}.
\end{proposition}

This inequality shows that there are concrete benefits to introducing $V(a)$.
For example, by a sustained and subtle argument,
Garsia (1970, 3.7.20, p.~92) arrived at a version of our inequality where the
coefficient $16 \{1+2V_n(a)\}$ is replaced by ${80}$. Now, for
uniform multipliers, one has $V_n(a)=O(n)$ and Garsia's
inequality is greatly superior to our bound \eqref{eq:GwithV}. On the other hand,
for multipliers that satisfy the cancelation property $V_n(a)\to 0$, the present bound eventually meets and beats Garsia's bound.
In particular, for multipliers given by alternating signs, the constant of \eqref{eq:GwithV} is just
$17+16/17$.

\section{Weighted Sums and Folding}

By many accounts, the permutation maximal inequality of Garsia (1970, p.~86) is the salient result in the theory of permutation
inequalities, so it is a natural challenge to see if it can be proved by the robust martingale methods
that follow from our martingale constructions. We give a proof of this kind --- without appeal to $V_n(a)$.
Once the proof is complete, we address the
differences between Garsia's inequality and the present bound with its curious constant.

\begin{proposition}\label{Garsia*Inequality*Weighted*Sums}
For real numbers $\{a_1,a_2,...,a_n\}$ and real numbers $\{x_1,x_2,...,x_n\}$ such that  $x_1+x_2+\cdots+x_n=0$ we have
\begin{equation}\label{eq:permWweights}
\frac{1}{n!}\sum_{\sigma}\max_{1\leq k\leq n}\big|\sum_{i=1}^ka_ix_{\sigma(i)}\big|^2
\leq
\big(80+\frac{4}{205}\big){\sum_{i=1}^na_i^2\sum_{i=1}^nx_i^2}\big/(n-1),
\end{equation}
where the sum is taken over all possible permutations of $\{x_1,x_2,...,x_n\}$.
\end{proposition}

\begin{proof}
First, for all $1\leq k \leq n$, Cauchy's inequality gives us
$\alpha_1^2(k)\leq k\alpha_2(k)$. Trivially one has $\alpha_2(k) \leq \alpha_2(n)$, so by Lemma \ref{lm:Msqrd} we have the bound
\begin{equation}\label{Max*Ineqiality*for*Weighted*Martingale}
E[M_k]^2\leq\bigg(1+\frac{k}{n-k}\bigg)\frac{\alpha_2(n) B}{n-1} \quad \quad \text{for } 1\leq k \leq n-1.
\end{equation}
Now, just from the definition \eqref{eq:WeightedMartingale2} of $M_k$, we can write  $W_k$ as
a difference between $M_k$ and $\alpha_1(k){S_k}/({n-k})$.
We can then use the crude bound
$(x+y)^2\leq 2x^2+2y^2$ and Cauchy's inequality to get
\begin{align*}
\max_{1\leq k\leq m}|W_k|^2&
                             \leq 2\max_{1\leq k\leq m}|M_k|^2+\alpha_1^2(k)\left|\frac{S_k}{n-k}\right|^2\\
                             &\leq 2\max_{1\leq k\leq m}|M_k|^2+2\max_{1\leq k\leq m}\alpha_1^2(k)\left|\frac{S_k}{n-k}\right|^2\\
                             &\leq 2\max_{1\leq k\leq m}|M_k|^2+(2m) \alpha_2(n) \max_{1\leq k\leq m}\left|\frac{S_k}{n-k}\right|^2.
\end{align*}
Long ago, in \eqref{S_2*Expectation},  we calculated $ES_m^2$, so here we can apply the $L^2$ maximal inequality to both
martingales $\{S_k/(n-k)\}$ and $\{M_k\}$ to get the bound \color{black}
\begin{equation}\label{Intermediate*W_k^2*Inequality}
E\max_{1\leq k\leq m}W_k^2\leq 8\left(1+\frac{m}{n-m}+\frac{m^2}{n(n-m)}\right)\frac{\alpha_2(n)B}{n-1}.
\end{equation}
Such an inequality for $1\leq m< n$ suggests the possibility of folding. To pursue this we first note
\begin{align}
\max_{1\leq k\leq n} W_k^2&= \max\left[\max_{1\leq k\leq m}W_k^2,\, \, \max_{m< k\leq n}W_k^2\right] \notag\\
                           &= \max\left[\max_{1\leq k\leq m}W_k^2,\, \, \max_{m< k\leq n}|W_m+a_{m+1}X_{m+1}+\cdots+a_kX_k|^2\right] \notag\\
                           &\leq \max\left[\max_{1\leq k\leq m}W_k^2,\,\, \, \, 2W_m^2+2\max_{m< k\leq n}|a_{m+1}X_{m+1}+\cdots+a_kX_k|^2\right]\notag\\
                           &\leq 2\left[\max_{1\leq k\leq m} W_k^2+\,\, \max_{m< k\leq n}|a_{m+1}X_{m+1}+\cdots+a_kX_k|^2\right].\label{eq:both}
\end{align}
By exchangeability, one expects the second maximum has a bound like the one derived in
(\ref{Intermediate*W_k^2*Inequality}). To make this explicit, one simply needs to
replace $m$ by $n-m$ in the upper bound of (\ref{Intermediate*W_k^2*Inequality}). Doing so gives us the sister bound
\begin{equation*}
E\max_{m< k\leq n}|a_{m+1}X_{m+1}+\cdots+a_kX_k|^2\leq8\left(1+\frac{(n-m)}{m}+\frac{(n-m)^2}{nm}\right)\frac{\alpha_2(n)B}{n-1}.
\end{equation*}
By our bounds on the two addends of \eqref{eq:both}, we then have
$$
E\max_{1\leq k\leq n}W_k^2\leq 16\left(2+ \frac{m}{n-m}+\frac{n-m}{m}+\frac{m^2}{n(n-m)}+\frac{(n-m)^2}{nm}\right)\frac{\alpha_2(n)B}{n-1}.
$$

It only remain to take $m= \lfloor n/2 \rfloor$ and to attend honestly to the consequences. If $n$ is even,
the constant that multiplies $\alpha_2(n)B/(n-1)$
is exactly $80$. For odd $n$, the constant approaches $80$ from above as $n$ increases to infinity,
and for $n=81$ the constant is $80+4/205$.
Furthermore, Cauchy's inequality gives us
$$
W_k^2 \leq \alpha_2(k) B \leq \alpha_2(n) B \quad \text{for all } 1\leq k \leq n,
$$
and we have $\alpha_2(n) B\leq (80+4/205)\alpha_2(n)B/(n-1)$ for all $n\leq 81$. So, in the end, we come to
\eqref{eq:permWweights}, our permutation maximal inequality with general weights.
\end{proof}

We would greet this result with some fanfare except that Garsia (1970, 3.7.20, p.~92) gives this bound with the constant $80$.
Still, one needs to keep in mind that we have pursued this derivation only to illustrate the usefulness of the martingales that are
given by our general linear algebraic construction. Perhaps it is victory enough to come so close to a long-standing result that was
originally obtained by a delicate  problem specific, argument.

Compared to Garsia's argument, the proof of \eqref{eq:permWweights} is straightforward. It is also reasonably robust and potentially capable of
further development, even though there seems to be no room to improve the constant.
In spirit the proof is close to the elegant argument of Stout (1974, pp. 145--148)
for his version of the easier unweighted inequality (Proposition~\ref{Garsia*Inequality}).
In each instance, the heart of the matter is the
application of the maximal $L^2$ inequality to some martingale. Here we have the benefit of ready access to the
martingales \eqref{eq:MartingalsM2M3} and \eqref{eq:WeightedMartingale} that were served up to us by our general construction.

\section{Observations and Connections}

Our focus here is on methodology, and our primary aim
has been to demonstrate the usefulness of a linear algebraic method for constructing martingales from the basic materials
of sampling without replacement. Through our examples
we hope to have shown that the martingales given by our general construction have honest bite.
In particular, these martingales yield reasonably direct proofs of a variety of permutation inequalities --- both new ones and old ones.

Among our new inequalities, the simple Hardy-type inequality \eqref{eq:SerflingInequality} seems particularly
attractive. If we had to isolate a single open problem for attention, then our choice would be to determine if the
constant of inequality \eqref{eq:SerflingInequality} is best possible. This problem seems feasible, but one will not get any help
from the easy arguments that show that the corresponding constant in Hardy's inequality is best possible.

The other new inequality that seems noteworthy is the Garsia-type inequality \eqref{Garsia*alternating*sums*proposition*two}
where we introduce $V(a)$, the measure of multiplier cancelations. This inequality may be long-winded, but it isolates a
common situation where one can do substantially better than the classic Garsia bound \eqref{eq:permWweights}.
The quadratic permutation inequality (Proposition \ref{our*martingale})
and the discrete bridge inequality \eqref{Now*Done} are more specialized,
and they may have a hard time finding regular employment. Still, they are perfect for the right job, and they also illustrate the
diversity of the martingales that come from the general construction.

We have developed several results in theory of permutation
inequalities to test the effectiveness of the permutation martingales given by our construction, but
permutation inequalities have a charm of their own, and one could always hope to do more.
We have already mentioned the remarkable maximal inequalities of
Chobanyan (1994) and Chobanyan and Salehi (2001) that exploit combinatorial mapping arguments in addition to martingale arguments.
It would be interesting to see if our new martingales could help more in that context.
Finally, we did not touch on the important weak-type (or Levy-type) permutation inequalities such as those studied in Pruss (1998) and
Levental (2001), but it seems reasonable to expect that the martingales \eqref{eq:MartingalsM2M3} and \eqref{eq:WeightedMartingale}
could also be useful in the theory of weak-type inequalities.

\end{document}